\documentclass[11pt]{amsart}   % includes 11pt font size (default 10pt)

\usepackage{verbatim}
\usepackage{amssymb,amsthm,amsmath,amsfonts}
\usepackage{color}
\definecolor{darkblue}{rgb}{0, 0, .4}
\definecolor{grey}{rgb}{.7, .7, .7}

\usepackage{ifpdf}
\ifpdf
  \usepackage[pdftex]{epsfig}

  \usepackage{lscape}

  \usepackage[breaklinks]{hyperref}
  \hypersetup{
            colorlinks=true,
            linkcolor=darkblue,
            anchorcolor=darkblue,
            citecolor=darkblue,
            pagecolor=darkblue,
            urlcolor=darkblue,
            pdftitle={},
            pdfauthor={}
  }
\else
  \usepackage[dvips]{epsfig}
  \usepackage{lscape}

  \newcommand{\href}[2]{#2}
  \newcommand{\url}[2]{#2}
\fi

\newtheorem{theorem}{Theorem}[section]
\newtheorem{lemma}[theorem]{Lemma}

\theoremstyle{definition}
\newtheorem{definition}[theorem]{Definition}
\newtheorem{example}[theorem]{Example}

\theoremstyle{remark}

\numberwithin{equation}{section}

\theoremstyle{theorem}
\newtheorem{corollary}[theorem]{Corollary}
\newtheorem{algorithm}[theorem]{Algorithm}
\newtheorem{proposition}[theorem]{Proposition}

\usepackage{graphicx}
\input xy 
\xyoption{all}

\usepackage{tikz}
\usetikzlibrary{trees}
\usepackage{ifthen}

\usetikzlibrary{patterns}

\newcommand{\I}{\underline{\mathcal{I}}}
\newcommand{\TB}{\underline{\mathcal{B}}}
\newcommand{\T}{\mathcal{T}}
\renewcommand{\S}{\mathcal{S}}
\newcommand{\B}{\mathcal{B}}
\renewcommand{\c}{\circ}
\newcommand{\pf}[1]{\pi|_{[#1]}}

\newcommand{\C}[0]{\mathcal{C}}

\newcommand{\pc}[1]{{\bf \bar{#1}}}
\newcommand{\po}[1]{{\mathrm{#1}}}

\begin{document}

\title{Avoiding patterns and making the best choice}

\begin{abstract}
We study a variation of the game of best choice (also known as the secretary problem or game of googol) under an additional assumption that the ranks of interview candidates are restricted using permutation pattern-avoidance.  We develop some general machinery for investigating interview orderings with a non-uniform rank distribution, and give a complete description of the optimal strategies for the pattern-avoiding games under each of the size three permutations.  The optimal strategy for the ``disappointment-free'' (i.e.  $321$-avoiding) interviews has a form that seems to be new, involving thresholds based on value-saturated left-to-right maxima in the permutation.
\end{abstract}

\author{Brant Jones}
\address{Department of Mathematics and Statistics, MSC 1911, James Madison University, Harrisonburg, VA 22807}
\email{\href{mailto:jones3bc@jmu.edu}{\texttt{jones3bc@jmu.edu}}}
\urladdr{\url{http://educ.jmu.edu/\~jones3bc/}}

%\keywords{secretary problem; pattern-avoidance}

\date{\today}

\maketitle

%%%%%%%%%%%%%%%%%%%%%%%%%%%%%%%%%%%%%%%%%%%%%%%%%%%%%%%%%%%%%%%%%%%%%
%  Section
%%%%%%%%%%%%%%%%%%%%%%%%%%%%%%%%%%%%%%%%%%%%%%%%%%%%%%%%%%%%%%%%%%%%%
\bigskip
\section{Introduction}\label{s:intro}

The game of best choice has been considered under many different names by
researchers with a wide variety of perspectives.  In the classical story, a
player conducts {\bf interviews} with a fixed number $N$ of {\bf candidates}.
After each interview, the player ranks the current candidate against all of the
candidates that have previously been considered (without ties).  The
interviewer must then decide whether to {\bf accept} the current candidate and
end the game or, alternatively, whether to {\bf reject} the current candidate
forever and continue playing in the hope of obtaining a better candidate in the
future.  (These rules model a ``tight'' market in which each candidate has many
options for employment and will not be available to be recalled for a second
interview later in the game.)

Various facets and extensions of this model can be investigated but the most
developed line of inquiry is to describe a strategy that maximizes the player's
chance of {\em hiring the candidate ranked best among the $N$ candidates}.
Notably, the classical analysis of this game assumes that all $N!$
orderings of rankings into interviews are equally likely.  It turns out that the
form of the optimal strategy is to reject an initial set of candidates and use
them as a training set by hiring the next candidate who is better than all of
them (or the last candidate if no subsequent candidate is better).  The question
then becomes when to make the transition from rejection to hiring: If the
training set is small, it is likely that our standards will be set too low
to capture the best candidate; if it is large, it is likely that it already
contains the best candidate who will be interviewed and rejected.  After some
analysis, it turns out that the asymptotically optimal transition point is after
we have rejected the first $1/e \approx 37\%$ of candidates; the probability of
success with this strategy is also $1/e$ as $N$ tends to infinity.

Although this analysis is mathematically pithy, involving the constant $e$ in a
surprising way, we believe the assumption that all $N!$ interview orderings are
equally likely is ultimately unrealistic.  Over the period that the player is
conducting the $N$ interviews, there can be both extrinsic trends in the
candidate pool as well as intrinsic learning on the part of the interviewer.
As the interviewer ranks the current candidates at each step, they aquire
information about the domain that should allow them to hone the pool to include
more relevant candidates at future time steps.  Overall, this results in
candidate ranks that are improving over time (rather than uniform).  We
establish in this paper two new models that produce interesting behavior in
this direction.

To describe them, we employ pattern-avoidance in order to restrict the
interview orderings.  This is a natural technique from the viewpoint of
algebraic combinatorics although its application for the best choice game seems
to be new.  Throughout this paper, we model interview orderings as {\bf
permutations}.  The permutation $\pi$ of $N$ is expressed in {\bf one-line
notation} as $[\pi_1 \pi_2 \cdots \pi_N]$ where the $\pi_i$ consist of the
elements $1, 2, \ldots, N$ (so each element appears exactly once).  In the best
choice game, $\pi_i$ is the rank of the $i$th candidate interviewed {\em in
reality}, where rank $N$ is best and $1$ is worst.  What the player sees at
each step, however, are {\em relative} rankings.  For example, corresponding to
the interview order $\pi = [2516374]$, the player sees the sequence of
permutations \[ 1, 12, 231, 2314, 24153, 241536, 2516374 \] and must use only
this information to determine at which step to transition.  The {\bf
left-to-right maxima} of $\pi$ consist of the elements $\pi_i$ that are larger
in value than all elements $\pi_j$ lying to the left.  For example, in the
interview order $\pi = [2516374]$ the best candidate occurs in the sixth
position and it is not difficult to see that we will successfully hire them if
and only if we transition from rejection to hiring between the last two
left-to-right maxima (i.e. after the fourth or fifth interview).

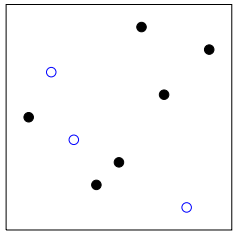
\begin{figure}[h]
\[ \scalebox{0.3}{ \begin{tikzpicture}
		% 584239617
		\draw (0,0) -- (10,0) -- (10,10) -- (0,10) -- (0,0);
		\draw (1,5) node [circle, fill=black, inner sep=0pt, minimum size=12pt, draw] {} ;
		\draw (2,7) node [circle, inner sep=0pt, minimum size=12pt, blue, draw] {} ;
		\draw (3,4) node [circle, inner sep=0pt, minimum size=12pt, blue, draw] {} ;
		\draw (4,2) node [circle, fill=black, inner sep=0pt, minimum size=12pt, draw] {} ;
		\draw (5,3) node [circle, fill=black, inner sep=0pt, minimum size=12pt, draw] {} ;
		\draw (6,9) node [circle, fill=black, inner sep=0pt, minimum size=12pt, draw] {} ;
		\draw (7,6) node [circle, fill=black, inner sep=0pt, minimum size=12pt, draw] {} ;
		\draw (8,1) node [circle, inner sep=0pt, minimum size=12pt, blue, draw] {} ;
		\draw (9,8) node [circle, fill=black, inner sep=0pt, minimum size=12pt, draw] {} ;
		%\draw (5.5,-1) node {position};
		%\draw (-1.5,5.5) node {value};
\end{tikzpicture} }\]
\caption{$\pi = [574239618]$} \label{f:ex}
\end{figure}

Given a permutation $\pi$ of $N$ and a permutation $\rho$ of $M \leq N$, we say
that $\pi$ {\bf contains $\rho$ as a pattern} if there exists a subsequence
$\pi_{i_1} \pi_{i_2} \cdots \pi_{i_M}$ of entries from $\pi$ in the same
relative order as the entries $\rho_1 \rho_2 \ldots \rho_M$ of $\rho$.
Otherwise, we say $\pi$ {\bf avoids} $\rho$.  This has a simple geometric
interpretation when we plot our permutations graphically by placing a point
$(i,j)$ in the Cartesian plane to represent $\pi_i = j$.  In Figure~\ref{f:ex},
we show the plot of $[574239618]$ and have highlighted a $321$-instance in
positions $2, 3$, and $8$.  As the example illustrates, neither the positions
nor the values in a pattern instance need to be consecutive.  This permutation
avoids $54321$.

We call the best choice game {\bf restricted} if we play on some subset of the
$N!$ interview orderings to obtain a distribution with, for simplicity,
probability zero on interview orderings that are not in the subset and uniform
probability on the orderings that are in the subset.  The choice of restriction
criteria represents the modeler's beliefs about the overall effect of the player
learning process on the interview orderings.  However, the player does not
directly impose these restrictions from within the game.

In our first model, we avoid the permutation $231$ which requires a sequence of
interviews to be {\bf status-seeking} in the sense that each time we have an
interview that is an improvement over some previous interview, a floor is set 
for all future candidate rankings.  Less drastically, one could imagine
avoiding $[2 3 \cdots (k-1) k 1] $ for various values of $k$ (tending towards
the classical game as $k \rightarrow \infty$), but the case we consider is the
strongest nontrivial model in this direction.  In a variation, avoiding $321$
requires interviews to be {\bf disappointment-free} in the sense that each time
we have an interview that is worse than some previous interview, the new
interview becomes a floor for all future candidate rankings.  These are
illustrated in Figure~\ref{f:strats}.  As above, one could extend this to
consider a family of patterns of the form $[k (k-1) \cdots 2 1]$ for various
$k$.

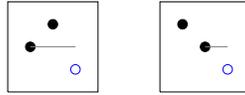
\begin{figure}[h]
\[ 
    \scalebox{0.3}{ \begin{tikzpicture}
			% 231
			\draw (0,0) -- (4,0) -- (4,4) -- (0,4) -- (0,0);
			\draw (1,2) node [circle, fill=black, inner sep=0pt, minimum size=12pt, draw] {} ;
			\draw (2,3) node [circle, fill=black, inner sep=0pt, minimum size=12pt, draw] {} ;
			\draw (3,1) node [circle, inner sep=0pt, minimum size=12pt, blue, draw] {} ;
			\draw [gray] (1,2) -- (3,2);
			%\draw (2.5,-1) node {position};
			%\draw (-1.5,2.5) node {value};
			\end{tikzpicture} } 
\hspace{0.3in}
		\scalebox{0.3}{ \begin{tikzpicture}
			% 321
			\draw (0,0) -- (4,0) -- (4,4) -- (0,4) -- (0,0);
			\draw (1,3) node [circle, fill=black, inner sep=0pt, minimum size=12pt, draw] {} ;
			\draw (2,2) node [circle, fill=black, inner sep=0pt, minimum size=12pt, draw] {} ;
			\draw (3,1) node [circle, inner sep=0pt, minimum size=12pt, blue, draw] {} ;
			\draw [gray] (2,2) -- (3,2);
			%\draw (2.5,-1) node {position};
			%\draw (-1.5,5.5) node {value};
\end{tikzpicture} } \]
\caption{Defining patterns for status-seeking and disappointment-free interview
orders}\label{f:strats}
\end{figure}

It is a surprising consequence of our analysis that although these patterns
have symmetric descriptions in terms of floor setting, they turn out to have
very different optimal strategies and asymptotic success probabilities!
Namely, the status-seeking interviews have simple and robust optimal strategies
that essentially allow the player to transition from rejection to hiring {\em
at any time}.  This results in a probability of success that is a ratio of
consecutive Catalan numbers which approaches $25\%$ as $N \rightarrow \infty$
(see Theorem~\ref{t:231_asy_main}).  Observe that this model allows a version
of the game in which no candidate prefers another interview position to their
own!  (By contrast, it is a recent criticism of the classical model that no
candidate would prefer to be among the first $1/e$ interview positions as they
will be rejected automatically.)

The disappointment-free interviews have a more subtle set of optimal strategies
involving thresholds based on {\em value-saturated} left-to-right maxima (see
Corollary~\ref{c:321_strat}) but the probability of success approaches a limit
that is more than $50\%$ (see Corollary~\ref{c:321_main}).  Note that the
optimal strategies for these models are not mutually exclusive so one may wager
{\em \`a la Pascal} and implement both at the same time; the optimal
disappointment-free strategy is also optimal for the status-seeking model.

In Robbins' problem (see \cite{bruss}) and other variants of the game where the
player seeks to maximize expected value, researchers develop ad hoc strategies
for analysis (because it is not clear what form an optimal strategy should
take).  As far as we know, thresholds based on saturated left-to-right maxima
have not appeared previously in the literature, but perhaps bear further
investigation in light of the fact that they yield the optimal strategy for our
disappointment-free interview orders.

We now mention some further ties to earlier work.  Martin Gardner's 1960
Scientific American column popularized what he called ``the game of googol,''
although the problem has roots which predate this.  His article has been
reprinted in \cite{gardner}.  One of the first papers to systematically study
the game of best choice in detail is \cite{gilbert--mosteller}.  Many other
variations and some history have been given in \cite{ferguson} and
\cite{freeman}.  Recently, researchers (e.g.  \cite{kleinberg08}) have begun
applying the best-choice framework to online auctions where the ``candidate
rankings'' are bids (that may arrive and expire at different times) and the
player must choose which bid to accept, ending the auction.  

Although there is an established ``Cayley'' or ``full-information'' version of
the game in which the player observes {\em values} from a given distribution,
it seems that only a few papers have considered alternative {\em rank} distributions
directly.  The paper \cite{reeves--flack} considers an explicit continuous
probability distribution that allows for dependencies between nearby arrival
ranks via a single parameter.  Inspired by approximation theory, the paper
\cite{kleinberg--etal} also studies some general properties of non-uniform rank
distributions for the secretary problem.  

Recent work of Buchbinder et al.  \cite{buchbinder--etal} uses linear
programming to find algorithms for solving best choice problems motivated by
applications to online auctions.  They give an ``incentive-compatible''
mechanism for which the probability of selecting a candidate is independent of their
position, a property that is closely related to the main finding for our
status-seeking game.  Under the rule that candidates may not be recalled once
rejected, their strategy is shown to succeed with probability $1/4$, exactly
the same as our asymptotic result for the status-seeking interviews!  At the
moment, this seems to be a curious coincidence.  

From the algebraic combinatorics perspective, Wilf has collected some results on
distributions of left-to-right maxima in \cite{wilf} and Prodinger
\cite{prodinger} has studied these under a geometric random model.  Although we
phrase our results in terms of the game of best choice, they may also be viewed
in some cases as an extension of the literature on distributions of
left-to-right maxima to subsets of pattern-avoiding permutations.  More
recently, several authors have investigated the distribution of various
permutation statistics for a random model in which a pattern-avoiding
permutation is chosen uniformly at random.  For example, \cite{miner-pak} finds
the positions of smallest and largest elements as well as the number of fixed
points in a random permutation avoiding a single pattern of size 3;
\cite{mp-312} finds the probability that one or two specified points occur in a
random permutation avoiding 312; and the work of several authors \cite{dhw,fmw}
determines the lengths of the longest monotone and alternating subsequences in a
random permutation avoiding a single pattern of size 3.  We also consider
uniformly random $321$-avoiding and $231$-avoiding permutations in our work, but
the statistics we are concerned with arise from the game of best choice.  In
some sense, our results refine the question of where a uniformly random
pattern-avoiding permutation achieves its maximum because in our problem we want
to transition so as to capture the maximum value.

We now outline the rest of this paper.  In Section~\ref{s:os}, we introduce the
notion of strike sets that express optimal strategies when we restrict our
interview orderings to some subset of permutations, and recall basic properties
of the Catalan and ballot numbers that serve as denominators for our
probabilities.  The rest of the paper characterizes the optimal strategies that
arise when we avoid a single pattern of size $3$.  In Sections~\ref{s:231} and
\ref{s:321} we analyze the status-seeking and disappointment-free models,
respectively.  We consider the remaining patterns of size $3$ in
Section~\ref{s:conclusions}.

%%%%%%%%%%%%%%%%%%%%%%%%%%%%%%%%%%%%%%%%%%%%%%%%%%%%%%%%%%%%%%%%%%%%%
%  Section
%%%%%%%%%%%%%%%%%%%%%%%%%%%%%%%%%%%%%%%%%%%%%%%%%%%%%%%%%%%%%%%%%%%%%
\bigskip
\section{Strikes and Strategies}\label{s:os}

Fix a subset $\I_N$ of permutations of size $N$.  Such a subset defines a
restricted game of best choice as follows. 

\begin{definition}\label{d:pflat}
Given a sequence of $i$ distinct integers, we define its {\bf flattening} to be
the unique permutation of $\{1, 2, \ldots, i\}$ having the same relative order
as the elements of the sequence.  Given a permutation $\pi$, define the {\bf
$i$th prefix flattening}, denoted $\pf{i}$, to be the permutation obtained by
flattening the sequence $\pi_1, \pi_2, \ldots, \pi_i$.
\end{definition}

In the restricted game of best choice, some $\pi \in \I_N$ is chosen (uniformly
randomly, with probability $1/|\I_N|$) and each prefix flattening $\pf{1},
\pf{2}, \ldots$ is presented sequentially to the player.  If the player stops
at value $N$, they win; otherwise, they lose.  

To describe the form of an optimal strategy for such games, form the {\bf
prefix tree} consisting of all possible prefixes partially ordered by the
prefix flattenings they contain.  For example, the complete tree for $N = 4$ is
shown in Figure~\ref{f:unrest}.  The $N$th level always includes all of the
actual interview orders from $\I_N$ that may be encountered.

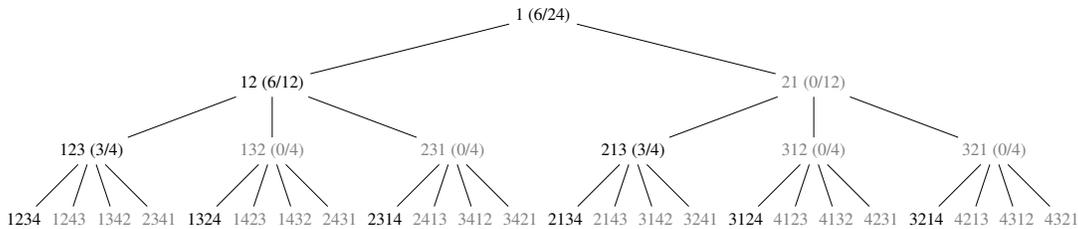
\begin{figure}[h]
\[ \scalebox{0.6}{ \begin{tikzpicture}[grow=down]
     %\tikzstyle{level 1}=[sibling distance=60mm]
     %\tikzstyle{level 2}=[sibling distance=20mm]
     %\tikzstyle{level 3}=[sibling distance=5mm]
     \tikzstyle{level 1}=[sibling distance=120mm]
     \tikzstyle{level 2}=[sibling distance=40mm]
     \tikzstyle{level 3}=[sibling distance=10mm]
    \node {1 (6/24)}
     child{ node {12 (6/12)} 
       child{ node {123 (3/4)}
       child{ node{1234}}
       child{ node[gray] {1243}}
       child{ node[gray] {1342}}
       child{ node[gray] {2341}}
   }
   child{ node[gray]  {132 (0/4)}
       child{ node{1324}}
       child{ node[gray] {1423}}
       child{ node[gray] {1432}}
       child{ node[gray] {2431}}
   }
   child{ node[gray]  {231 (0/4)}
       child{ node{2314}}
       child{ node[gray] {2413}}
       child{ node[gray] {3412}}
       child{ node[gray] {3421}}
   }
 }
 child{ node[gray]  {21 (0/12)} 
       child{ node {213 (3/4)}
       child{ node{2134}}
       child{ node[gray] {2143}}
       child{ node[gray] {3142}}
       child{ node[gray] {3241}}
   }
   child{ node[gray]  {312 (0/4)}
       child{ node{3124}}
       child{ node[gray] {4123}}
       child{ node[gray] {4132}}
       child{ node[gray] {4231}}
   }
   child{ node[gray]  {321 (0/4)}
       child{ node{3214}}
       child{ node[gray] {4213}}
       child{ node[gray] {4312}}
       child{ node[gray] {4321}}
   }
 };
\end{tikzpicture}
}
\]
\caption{Unrestricted prefix tree for $N=4$ with strike probabilities}\label{f:unrest}
\end{figure}

A {\bf strike strategy} for a restricted game of best choice is defined by a
collection of prefixes we call the {\bf strike set}.  To play the strategy on a
particular interview ordering $\pi$, compare prefix flattenings to the strike
set at each step.  As soon as the $i$th prefix flattening occurs in the strike
set, accept the candidate at position $i$ and end the game.  Otherwise, the
strike strategy rejects the candidate at position $i$ to continue playing.

It follows directly from the definitions that any strategy (including the
optimal strategy) for a game of best choice can be represented as a strike
strategy because the player has only the relative ranking information captured
in the prefix flattenings to guide them as they play.  It suffices to restrict
our attention to strike sets that are antichains, meaning that no pair of
elements are related in the prefix tree.  If the descendants of a strike set
eventually include every permutation of size $N$, we say the strategy is {\bf
complete}.  Given a subset $S$ of prefixes, the {\bf completion} of $S$ is the
strike set consisting of the prefixes from $S$ together with all the
permutations from $\I_N$ that are not descendants of any prefix in $S$.

For each prefix $p$, let the {\bf strike probability} $\S_N(p)$ represent the
probability of winning the game if $p$ is included in the strike set for the
subset of interview orderings containing $p$ as a prefix.  Explicitly for $p =
p_1 p_2 \cdots p_k$, say that $\pi$ is {\bf $p$-winnable} if $\pi$ has $p$ as a
prefix flattening and value $N$ in the $k$th position.  Then $\S_N(p)$ is the
number of $p$-winnable permutations divided by the total number of permutations
from $\I_N$ having $p$ as a prefix flattening.  Note this implies that $\S_N(p)$
is $0$ unless $p$ ends in a left-to-right maximum.  For this reason, we refer to
a prefix as {\bf eligible} if it ends in a left-to-right maximum.  We have
indicated the $\S_N(p)$ values for $N = 4$ in Figure~\ref{f:unrest} with
ineligible prefixes shown in gray.

Given a complete strike set $A$, the probability of winning using the
corresponding strike strategy is then
\[ \mathcal{P}(A) = \bigoplus_{p \in A} \S_N(p) \]
where
\begin{enumerate}
    \item[(a)]  We view the probabilities $\S_N(p)$ as {\em pairs of integers} given by
        the numerator (number of $p$-winnable permutations) and denominator
        (total number of permutations having $p$ as a prefix flattening), 
        not as rational numbers.
    \item[(b)]  We sum $\frac{a}{b}$ and $\frac{c}{d}$ as the probability of the
        union of their underlying (independent) events, so $\frac{a}{b} \oplus \frac{c}{d} =
        \frac{a+c}{b+d}$.  
\end{enumerate} 

We say that a complete strike set $A$ is {\bf optimal} if $\mathcal{P}(A)$ is
maximal among all complete strike sets $A$.  Observe that the following
``backwards induction'' algorithm will always find an optimal strike set $A$.  

\begin{algorithm}\label{a:optimal}
Begin with $A$ equal to $\I_N$.  Choose an eligible prefix $p$ of maximal
size that has not yet been considered as part of this algorithm.  If $\S(p)$
is larger than the probability of winning on the interviews restricted to the
subforest strictly below the $p$ in the prefix tree, playing under the best
strategy that has been obtained so far, then replace the elements of $A$ having
$p$ as a prefix with $p$.  Otherwise, continue on to the next unconsidered
eligible prefix.  Eventually, we will consider the prefix $[1]$ at which point
$A$ will be a globally optimal strike set.
\end{algorithm}

\begin{example}
By inspection, the optimal strike set for the unrestricted set of interview
orders in $N=4$ is the completion of $\{ 12, 213, 3124, 3214 \}$, contributing
$11$ winners.  (This strategy coincides with the strategy that rejects
the first candidate and selects the next left-to-right maximum thereafter.)
\end{example}

In the unrestricted best choice game, the $\S_N(p)$ probability for an eligible
prefix $p$ is equal to the $\S_N(p')$ probability for any other eligible prefix
$p'$ of the same size.  This follows because we can perform an automorphism of
the symmetric group that rearranges all of the permutations with prefix
flattening $p$ to have prefix flattening $p'$.  Using arguments from
\cite{kadison}, the classical results can then be stated in terms of a {\bf
positional strategy} in which the player rejects the first $k$ candidates and
accepts the next left-to-right maximum thereafter.  

It is possible to generalize this to a {\bf trigger strategy} by defining a set
of prefixes to be {\bf triggers} in the sense that when a prefix flattening of
$\pi$ matches an element of the trigger set, the player rejects the current interview
candidate but accepts the next left-to-right maximum thereafter.  In the
classical case, optimal trigger sets consist of all prefixes having the critical
$1/e$ size.  

More generally still, we may consider some statistic from the set of prefixes to
nonnegative integers (say) together with a threshold function for that statistic
that can be used to define a {\bf threshold trigger} or {\bf threshold
strike} strategy.  That is, once the statistic is larger than the threshold, 
accept the current interview candidate (strike) or transition to accept the next
left-to-right maximum (trigger).  For example, the size statistic together
with the $1/e$ trigger threshold defines the optimal strategy in the classical
case, and we will see in Section~\ref{s:321} that the number of value-saturated
left-to-right maxima together with a quadratic strike threshold function defines
the optimal strategy in the disappointment-free (i.e. $321$-avoiding) case.  It
is an interesting problem to characterize or determine properties of $\I_N$ that
restrict the optimal strategy to one of these classes.

Given two restriction criteria, we say that the resulting prefix trees are {\bf
isomorphic} if there exists a bijection from the nodes of one tree to the other
that preserves the tree structure as well as the $\S_N(p)$ values.  In this
situation, we then obtain isomorphic strike sets to describe optimal strategies
for the two games and the probabilities of success under optimal play will be
equal.  In this work, our restriction criteria come from permutation pattern
avoidance and we call two patterns {\bf best-choice Wilf equivalent} if they
induce the same optimal probability of success in each restricted game of best
choice for all $N$.  Clearly, when two patterns induce isomorphic prefix trees
they are best-choice Wilf equivalent (and we know of no other examples).  In
the subsequent sections of this paper, we will find that there are four of
these generalized Wilf equivalence classes in $\mathfrak{S}_3$: \[ 231 \cong 132, 321
\cong 312, 123, 213.\]

Define the integer sequence of {\bf Catalan numbers} $\C_N$ by 
\[ \C_N = \sum\limits_{i+j=N-1} \C_i \C_j = \C_0 \C_{N-1} + \C_1 \C_{N-2} + \cdots + \C_{N-1} \C_0 \]
where $\C_0 = 1$ and $\C_1 = 1$.  The first few terms are $1, 1, 2, 5, 14, 42,
132, 429, 1430, \ldots$, and we have the explicit formula
\[ \C_N = \frac{1}{N+1} {2N \choose N}. \]
One of the earliest enumerative results in permutation pattern avoidance is that
the number of permutations avoiding any pattern of size $3$ is counted by the
Catalan sequence, so these will form denominators in our probability
calculations.  We refer the reader to the textbook \cite{bona} for details and
references.

The {\bf ballot numbers} (sequence A009766 in \cite{oeis})
\[ \C(N,k) = \frac{k+1}{N+1} { {2N-k} \choose {N} } \]
are a refinement of the Catalan numbers (where $\C(N,0) = \C(N,1) = \C_N$).
Some data is shown in Figure~\ref{f:ballot}.  They are defined by (either of)
the recurrences
%\[ \C(N,k) = \sum_{i=k-1}^{N-1} \C(N-1,i) = \C(N-1,k-1) + \C(N,k+1) \]
\[ \C(N,k) = \sum_{i=0}^{k} \C(N-i,k+1-i) = \C(N-1,k-1) + \C(N,k+1) \]
with the initial conditions that $\C(N,N) = 1$ for all $N$.  
%An example of the two-term recurrence is highlighted in Figure~\ref{f:ballot}.

\begin{figure}[h]
    \scalebox{0.8}{
\begin{tabular}{llllllllllllll}
    $N$ & \color{gray} $k=0$ & $k=1$ & $k=2$ & $k=3$ & $k=4$ & $\cdots$ \\
    \hline \\
    $1$ & \color{gray} $1$ & $1$ \\
    $2$ & \color{gray} $2$ & ${ 2 }$ & $1$ \\
    $3$ & \color{gray} $5$ & ${ 5 }$ &  ${ 3 }$ & $1$ \\
    $4$ & \color{gray} $14$ & ${ 14 }$ &  ${ 9 }$ &  ${ 4 }$ &  $1$ \\
    $5$ & \color{gray} $42$ & ${ 42 }$ &  ${ 28 }$ &  ${ 14 }$ &  ${ 5 }$ &  $1$ \\
    $6$ & \color{gray} $132$ & ${ 132 }$ &  ${ 90 }$ &  ${ 48 }$ &  ${ 20 }$ &  ${ 6 }$ & $1$  \\
    $7$ & \color{gray} $429$ & ${ 429 }$ &  ${ 297 }$ &  ${ 165 }$ &  ${ 75 }$ &  ${ 27 }$ &  ${ 7 }$ & $1$ \\
    $8$ & \color{gray} $1430$ & ${ 1430 }$ &  ${ 1001 }$ &  ${ 572 }$ &  ${ 275 }$ &  ${ 110 }$ &  ${ 35 }$ & ${ 8 }$ & $1$ \\
    $9$ & \color{gray} $4862$ & ${ 4862 }$ &  ${ 3432 }$ &  ${ 2002 }$ &  ${ 1001 }$ &  ${ 429 }$ &  ${ 154 }$ & ${ 44 }$ &  ${ 9 }$ & $1$ \\
%$10$ & $16796$ & ${ 16796 }$ &  ${ 11934 }$ &  ${ 7072 }$ &  ${ 3640 }$ &  ${ 1638 }$ &  ${ 637 }$ &  ${ 208 }$ &  ${ 54 }$ &  ${ 10 }$ & $1$ \\
\end{tabular} }
\caption{Ballot numbers}\label{f:ballot}
\end{figure}

These ballot numbers arise as denominators of $\S_N(p)$ probabilities.

%%%%%%%%%%%%%%%%%%%%%%%%%%%%%%%%%%%%%%%%%%%%%%%%%%%%%%%%%%%%%%%%%%%%%
%  Section
%%%%%%%%%%%%%%%%%%%%%%%%%%%%%%%%%%%%%%%%%%%%%%%%%%%%%%%%%%%%%%%%%%%%%
\bigskip
\section{231-avoiding (isomorphic with 132-avoiding)}\label{s:231}

In this section, we derive optimal strategies for the 231-avoiding best choice
game, and show that this game is isomorphic to the 132-avoiding best choice
game.

\begin{lemma}\label{l:231_structure}
    Every $231$-avoiding permutation $\pi$ decomposes 
as $[\pi_1 \pi_2 \cdots \pi_{i-1} N \pi_{i+1} \cdots \pi_N]$
where each entry of $[\pi_1 \cdots \pi_{i-1}]$ is smaller in value than each entry
of $[\pi_{i+1} \cdots \pi_N]$, and both of these are $231$-avoiding.
\end{lemma}
\begin{proof}
    Graphically, we are claiming that the diagram of any $231$-avoiding permutation must have the form
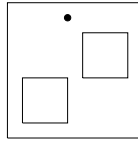
\[ \scalebox{0.2}{ \begin{tikzpicture}
		\draw (0,0) -- (9,0) -- (9,9) -- (0,9) -- (0,0);
		\draw (4,8) node [circle, fill=black, inner sep=0pt, minimum size=12pt, draw] {} ;
		\draw (1,1) -- (4,1) -- (4,4) -- (1,4) -- (1,1);
		\draw (5,4) -- (8,4) -- (8,7) -- (5,7) -- (5,4);
\end{tikzpicture} }\]
where each block is itself $231$-avoiding.  Observe that this decomposition
also realizes the Catalan recurrence. This is straightforward (because $N$
must play the role of $3$ in any $231$ instance) and well-known; see
\cite{bona} for example.
\end{proof}

The complete $231$-avoiding prefix tree for $N = 4$ is shown in
Figure~\ref{f:231tree} with the strike probabilities $\S_N(p)$ given in
parentheses.  We mention that these prefix trees (up to some position/value
conventions) coincide with the generating trees used by West to give uniform
proofs of some early enumerative results on permutation patterns; see
\cite{west} for an introduction.

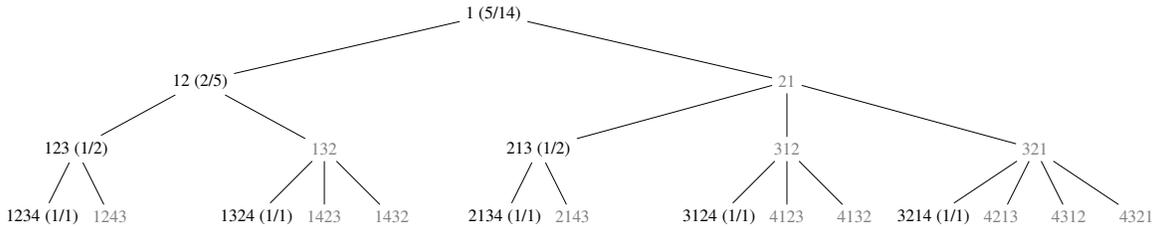
\begin{figure}[h]
\[ 
\scalebox{0.6}{ \begin{tikzpicture}[grow=down]
     %\tikzstyle{level 1}=[level distance=40mm, sibling distance=50mm]
     %\tikzstyle{level 2}=[level distance=40mm, sibling distance=20mm]
     %\tikzstyle{level 3}=[level distance=40mm, sibling distance=5mm]
     %\tikzstyle{level 1}=[sibling distance=100mm]
     \tikzstyle{level 1}=[sibling distance=130mm]
     \tikzstyle{level 2}=[sibling distance=55mm]
     \tikzstyle{level 3}=[sibling distance=15mm]
    \node {1 (5/14)}
     child{ node {12 (2/5)} 
       child{ node {123 (1/2)}
       child{ node[black] {1234 (1/1)}}
       child{ node[gray] {1243}}
   }
   child{ node[gray] {132}
       child{ node[black] {1324 (1/1)}}
       child{ node[gray] {1423}}
       child{ node[gray] {1432}}
   }
 }
 child{ node[gray] {21} 
       child{ node {213 (1/2)}
       child{ node[black] {2134 (1/1)}}
       child{ node[gray] {2143}}
   }
   child{ node[gray] {312}
       child{ node[black] {3124 (1/1)}}
       child{ node[gray] {4123}}
       child{ node[gray] {4132}}
   }
   child{ node[gray] {321}
       child{ node[black] {3214 (1/1)}}
       child{ node[gray] {4213}}
       child{ node[gray] {4312}}
       child{ node[gray] {4321}}
   }
 };
\end{tikzpicture}
}
\]
\caption{$231$-avoiding prefix tree for $N = 4$}\label{f:231tree}
\end{figure}

\begin{definition}
Given a $231$-avoiding permutation $\pi$ of size $N$, define $\varphi(\pi)$ as
follows.  Move the value $N$ (at least once) one position at a time to the
right, keeping all other values in the same relative order.  Stop as soon as
the permutation avoids $231$.  
\end{definition}

In practice, this means that $\varphi$ moves $N$ past any and all entries that
form an inversion with the entry immediately to the right of $N$.  Also note
that placing $N$ in the last position is always valid, so $\varphi$ is defined
for all $\pi$ not having $N$ in the last position.  

For our main result, we define the {\bf successors} of an eligible prefix $p$
to be the set of prefixes $q$ (eligible or not) whose longest proper eligible
prefix is $p$.

\begin{theorem}\label{t:231_main}
For each eligible prefix $p$ in the $231$-avoiding prefix tree, we have 
\[ \S_N(p) = \bigoplus_{\substack{\text{successors $q$ of $p$}}} \S_N(q). \]
\end{theorem}
\begin{proof}
From the definitions, it is clear that the denominator on the left side is equal
to the sum of the denominators on the right side.  Suppose $\pi$ is a
$231$-avoiding permutation of $N$ and suppose that $p$ is the prefix flattening
$\pf{i}$, which ends in a left-to-right maximum.  If including $p$ in the strike
set wins $\pi$ then $\pi_i = N$.  In this case, $\varphi(\pi)$ will be winnable
with $q$ in the strike set for precisely one successor $q$, namely $q =
\varphi(\pi)|_{[j]}$ where $j$ is the position of $N$ in $\varphi(\pi)$.
Conversely, if $q$ is an eligible successor of $p$ and $\pi$ is winnable with
$q$ in the strike set, then $\pi_j = N$ where $j$ is the size of $q$.  Sliding
$N$ back to position $i$, where $i$ is the size of $p$, and keeping the other
values of $\pi$ in the same relative order yields a permutation that is winnable
with $p$ in the strike set.  Since $\pi_i$ will be the left-to-right maximum
immediately prior to $\pi_j = N$ in $\pi$, the resulting permutation will remain
$231$-avoiding.  Thus, for a fixed eligible prefix $p$, we have that $\varphi$
is a bijection from the set of $p$-winnable permutations to the union of the set
of $q$-winnable permutations where $q$ is any eligible successor of $p$.  This
proves the numerator on the left side is equal to the sum of the numerators on
the right side.
\end{proof}

\begin{corollary}\label{c:231cc}
For the $231$-avoiding interviews, the completion of any antichain of eligible
prefixes forms an optimal strike set.
\end{corollary}
\begin{proof}
Consider the set of all $231$-avoiding prefixes of size $N$.  Any other
antichain in the prefix tree consisting of some eligible prefixes together with
some prefixes of size $N$ can be obtained from this by a sequence of moves in
which we replace a collection of successors with their predecessor, as in
Algorithm~\ref{a:optimal}.  By Theorem~\ref{t:231_main}, we do not change the
probability of success.  
\end{proof}

\begin{theorem}
Any complete trigger set is optimal.  In particular, any positional strategy is
optimal.
\end{theorem}
\begin{proof}
We claim that
\[ \T_N(p) = \bigoplus_{\substack{\text{children $q$ of $p$}}} \T_N(q). \]
where $\T_N(p)$ are the {\bf trigger probabilities} defined as the number of
$\pi$ that are won if $p$ is included as a trigger (explicitly, the number of
$\pi$ where the last entry of $p$ lies between the last two left-to-right
maxima in $\pi$) divided by the total number of $\pi$ having $p$ as a prefix
flattening.  We also allow a null prefix $p = \emptyset$; as a trigger, this
simply selects the first interview candidate.  The trigger probabilities are
illustrated for $N = 4$ in Figure~\ref{f:T231tree}.

To verify the equation, let $\pi$ be a $231$-avoiding permutation of $N$ and
suppose that $p$ is the prefix flattening $\pf{i-1}$ and $q$ is the prefix
flattening $\pf{i}$.  Then, there are several cases.
\begin{enumerate}
    \item[(1)]  Suppose $\pf{i}$ does not end in a left-to-right maximum.  Then
        $\pi$ is $p$-winnable if and only if it is $q$-winnable.  
    \item[(2)]  Suppose $\pf{i}$ does end in a left-to-right maximum.  Observe
        that $\varphi(\pi)|_{[i]}$ also ends in a left-to-right maximum since all
        entries to the right of $N$ are larger than entries to the left of $N$
        by Lemma~\ref{l:231_structure}.  
        
        Also note that in $\pi$, value $N$ must lie in some position $j \geq i$ (for
        otherwise we would be in Case (1) above because $N$ is always the last
        left-to-right maximum).

        Now consider the following subcases: 
    \begin{enumerate} 
        \item [(a)]  $\pi$ is $q$-winnable.  Then $j > i$.

        \item [(b)]  $\pi$ is not $q$-winnable but $\varphi(\pi)$ is $q$-winnable.
            Then we must have $j = i$.

        \item [(c)]  Neither $\pi$ nor $\varphi(\pi)$ are $q$-winnable.  Then
            neither element can be $p$-winnable because any $p$-winnable
            element in Case (2) must have $N$ in position $i$ and then
            it is straightforward to see that $\varphi(\pi)$ would be $q$-winnable.
    \end{enumerate}
    Observe that applying $\varphi$ is a bijection from the elements in subcase
    (b) to the elements in subcase (a).  Moreover, the elements in subcase (b)
    are $p$-winnable but not $q$-winnable, while elements in subcase (a) are
    $q$-winnable but not $p$-winnable.  The elements in subcase (c) are neither
    $p$-winnable nor $q$-winnable.
\end{enumerate}

Thus, the probabilities are preserved as claimed in each case.  Since any
trigger set can be obtained, starting from all of the size $N-1$ prefixes as our
initial trigger set (which agrees with the corresponding initial strike
strategy) by a sequence of moves in which we replace a collection of children
with their parent, we find that any complete trigger strategy (and hence any
positional strategy) is optimal.
\end{proof}

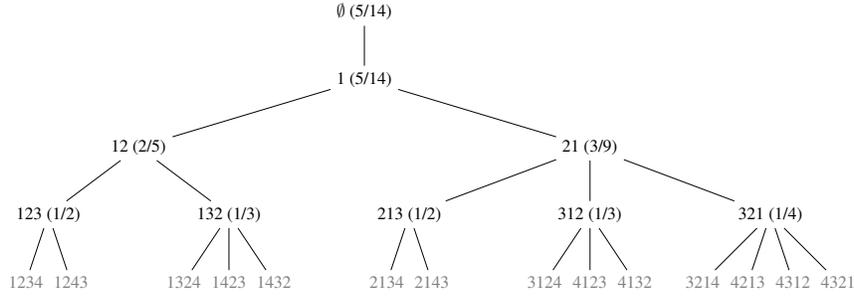
\begin{figure}[t]
\[ \scalebox{0.6}{ \begin{tikzpicture}[grow=down]
     \tikzstyle{level 1}=[sibling distance=100mm]
     \tikzstyle{level 2}=[sibling distance=100mm]
     \tikzstyle{level 3}=[sibling distance=40mm]
     \tikzstyle{level 4}=[sibling distance=10mm]
    \node {$\emptyset$ (5/14)}
    child{ node {1 (5/14)}
     child{ node {12 (2/5)} 
       child{ node {123 (1/2)}
       child{ node[gray] {1234}}
       child{ node[gray] {1243}}
   }
       child{ node {132 (1/3)}
       child{ node[gray] {1324}}
       child{ node[gray] {1423}}
       child{ node[gray] {1432}}
   }
 }
     child{ node {21 (3/9)} 
       child{ node {213 (1/2)}
       child{ node[gray] {2134}}
       child{ node[gray] {2143}}
   }
       child{ node {312 (1/3)}
       child{ node[gray] {3124}}
       child{ node[gray] {4123}}
       child{ node[gray] {4132}}
   }
       child{ node {321 (1/4)}
       child{ node[gray] {3214}}
       child{ node[gray] {4213}}
       child{ node[gray] {4312}}
       child{ node[gray] {4321}}
   }
 } };
\end{tikzpicture} } \]
\caption{Trigger probabilities for $N = 4$}\label{f:T231tree}
\end{figure}

\begin{theorem}\label{t:231_asy_main}
For the $231$-avoiding interviews of size $N$, the optimal probability of success is the
ratio of Catalan numbers $\C_{N-1}/\C_N$, which approaches $1/4$ as $N \rightarrow
\infty$.
\end{theorem}
\begin{proof}
Consider the strike set consisting of the $\C_N$ $231$-avoiding prefixes of
size $N$.  This is optimal by Corollary~\ref{c:231cc}.  Exactly $\C_{N-1}$ of
them are winnable because $N$ must lie in the last position.  Hence, the
numerator in the success probability for this strike set is $\C_{N-1}$.  Therefore,
the probability of success is $\C_{N-1} / \C_N$, and the asymptotics for this
ratio of Catalan numbers are straightforward from the explicit formula given in
Section~\ref{s:os}.
\end{proof}

Finally, we show that the $231$-avoiding prefix tree and $\S_N(p)$ probabilities
are isomorphic to those for the $132$-avoiding interview orders.

\begin{definition}
Let $\pi = [\pi_L \ \ N\ \  \pi_R]$ be a $231$-avoiding permutation with the
decomposition from Lemma~\ref{l:231_structure}, so the block $\pi_L$ has values
$\{1, 2, \ldots, k\}$ and $\pi_R$ has values $\{k+1, \ldots, N-1\}$.

Then we define $\Upsilon$ recursively where $\Upsilon$ is the identity on
permutations of size $2$ or less and in general,
\[ \Upsilon(\pi) = [ \Upsilon(\pi_L)_\uparrow \ \ N\ \ \Upsilon(\pi_R)_\downarrow ] \]
where the $\uparrow$ and $\downarrow$ operators reverse the blocks of values;
i.e. $\Upsilon(\pi_L)_\uparrow$ is a block with values $\{N-k, \dots, N-1\}$
and relative order given by $\Upsilon(\pi_L)$, and $\Upsilon(\pi_R)_\downarrow$
is a block with values $\{1, 2, \ldots, N-k-1\}$ and relative order given by
$\Upsilon(\pi_R)$.
\end{definition}

It is straightforward to check that $\Upsilon$ is a bijection from the set of
$231$-avoiding permutations of size $N$ to the set of $132$-avoiding
permutations of size $N$.

\begin{theorem}
The prefix tree for $231$-avoiding permutations is isomorphic to the prefix
tree for $132$-avoiding permutations.
\end{theorem}
\begin{proof}
Apply the bijection $\Upsilon$ to each $231$-avoiding prefix $p$.  We first
claim that the tree structure is preserved.  Namely, if $p$ is a prefix of $q$
in the $231$-avoiding tree then $\Upsilon(p)$ is a prefix of $\Upsilon(q)$ in
the $132$-avoiding tree.  We refer to this as the prefix property.  Note that
it suffices to check this when the size of $p$ is one less than the size of
$q$ by transitivity.  

Our strategy is to apply induction on the number of recursive steps used in the
application of $\Upsilon$.  So suppose that $p$ and $q$ are prefixes with $q =
[q_1 \cdots q_{i-1} q_i]$ where $q|_{[i-1]}$ and $p$ have the same relative
order.  If we apply $\Upsilon$ for a single recursive step, we will split $p$
and $q$ at their maximum elements, say $p_{MAX}$ and $q_{MAX}$.  If these
elements are equal, then the left factors of $p$ and $q$ have the same relative
order, while the right factors differ by the extra element $q_i$ at the end of
$q$.  Since applying $\Upsilon$ shifts the values as a block, we see that all
of the elements in the right factors of $q$ (not including $q_i$) and $p$ will
continue to have the same relative order.  So the prefix property holds by
induction in this case.  Otherwise, we must have $q_{MAX} = q_i$.  Then this
step has an empty right block in the decomposition of $q$.  So, in the
application of $\Upsilon$ at the {\em next} recursive step for $q$ we find that
$q_{MAX}$ will agree with $p_{MAX}$ in the application at {\em this} step for
$p$, so the prefix property continues to hold by induction in this case also.

Since applying $\Upsilon$ does not change the position of $N$, we have that
$\pi$ is $\pf{j}$-winnable if and only if $\Upsilon(\pi)$ is
$\Upsilon(\pi)|_{[j]}$-winnable.  So the $\S_N(p)$ probabilities are preserved
as well.
\end{proof}

%%%%%%%%%%%%%%%%%%%%%%%%%%%%%%%%%%%%%%%%%%%%%%%%%%%%%%%%%%%%%%%%%%%%%
%  Section
%%%%%%%%%%%%%%%%%%%%%%%%%%%%%%%%%%%%%%%%%%%%%%%%%%%%%%%%%%%%%%%%%%%%%
\bigskip
\section{321 avoiding (isomorphic with 312-avoiding)}\label{s:321}

\subsection{Introduction}
The best choice game restricted to the $321$-avoiding interview orders is a bit
more complicated.  The prefix trees for $N = 4$ and $N = 5$ (the latter omits
the last level) with the $\S_N(p)$ probabilities are shown in
Figure~\ref{f:321trees}.

\begin{figure}[h]
\[ 
\scalebox{0.6}{ \begin{tikzpicture}[grow=right]
     %\tikzstyle{level 1}=[level distance=40mm, sibling distance=50mm]
     %\tikzstyle{level 2}=[level distance=40mm, sibling distance=20mm]
     %\tikzstyle{level 3}=[level distance=40mm, sibling distance=5mm]
     %\tikzstyle{level 1}=[level distance=20mm, sibling distance=50mm]
     \tikzstyle{level 1}=[level distance=25mm, sibling distance=50mm]
     \tikzstyle{level 2}=[level distance=30mm, sibling distance=20mm]
     \tikzstyle{level 3}=[level distance=35mm, sibling distance=5mm]
     %\tikzstyle{level 1}=[sibling distance=100mm]
     %\tikzstyle{level 2}=[sibling distance=40mm]
     %\tikzstyle{level 3}=[sibling distance=10mm]
    \node {1 (1/14)}
     child{ node {12 (3/9)} 
       child{ node {123 (3/4)}
       child{ node[black] {1234 (1/1)}}
       child{ node[gray] {1243 (0/1)}}
       child{ node[gray] {1342 (0/1)}}
       child{ node[gray] {2341 (0/1)}}
   }
   child{ node[gray] {231 (0/3)}
       child{ node[black] {2314 (1/1)}}
       child{ node[gray] {2413 (0/1)}}
       child{ node[gray] {3412 (0/1)}}
   }
   child{ node[gray] {132 (0/2)}
       child{ node[black] {1324 (1/1)}}
       child{ node[gray] {1423 (0/1)}}
   }
 }
 child{ node[gray] {21 (0/5)} 
       child{ node {213 (2/3)}
       child{ node[black] {2134 (1/1)}}
       child{ node[gray] {2143 (0/1)}}
       child{ node[gray] {3142 (0/1)}}
   }
   child{ node[gray] {312 (0/2)}
       child{ node[black] {3124 (1/1)}}
       child{ node[gray] {4123 (0/1)}}
   }
 };
\end{tikzpicture}
}
\hspace{0.2in}
\scalebox{0.6}{ \begin{tikzpicture}[grow=right]
     %\tikzstyle{level 1}=[level distance=20mm, sibling distance=51mm]
     \tikzstyle{level 1}=[level distance=25mm, sibling distance=51mm]
     \tikzstyle{level 2}=[level distance=30mm, sibling distance=21mm]
     \tikzstyle{level 3}=[level distance=35mm, sibling distance=6mm]
     %\tikzstyle{level 1}=[sibling distance=130mm]
     %\tikzstyle{level 2}=[sibling distance=60mm]
     %\tikzstyle{level 3}=[sibling distance=20mm]
    \node {1 (1/42)}
     child{ node {12 (4/28)} 
       child{ node {123 (6/14)}
       child{ node {1234 (4/5)}}
       child{ node[gray] {1243 (0/2)}}
       child{ node[gray] {1342 (0/3)}}
       child{ node[gray] {2341 (0/4)}}
   }
   child{ node[gray] {231 (0/9)}
       child{ node {2314 (3/4)}}
       child{ node[gray] {2413 (0/2)}}
       child{ node[gray] {3412 (0/3)}}
   }
   child{ node[gray] {132 (0/5)}
       child{ node {1324 (2/3)}}
       child{ node[gray] {1423 (0/2)}}
   }
 }
 child{ node[gray] {21 (0/14)} 
       child{ node {213 (3/9)}
       child{ node {2134 (3/4)}}
       child{ node[gray] {2143 (0/2)}}
       child{ node[gray] {3142 (0/3)}}
   }
   child{ node[gray] {312 (0/5)}
       child{ node {3124 (2/3)}}
       child{ node[gray] {4123 (0/2)}}
   }
 } ;
\end{tikzpicture}
}
\]
\caption{$321$-avoiding prefix trees for $N = 4$ and $N = 5$}\label{f:321trees}
\end{figure}
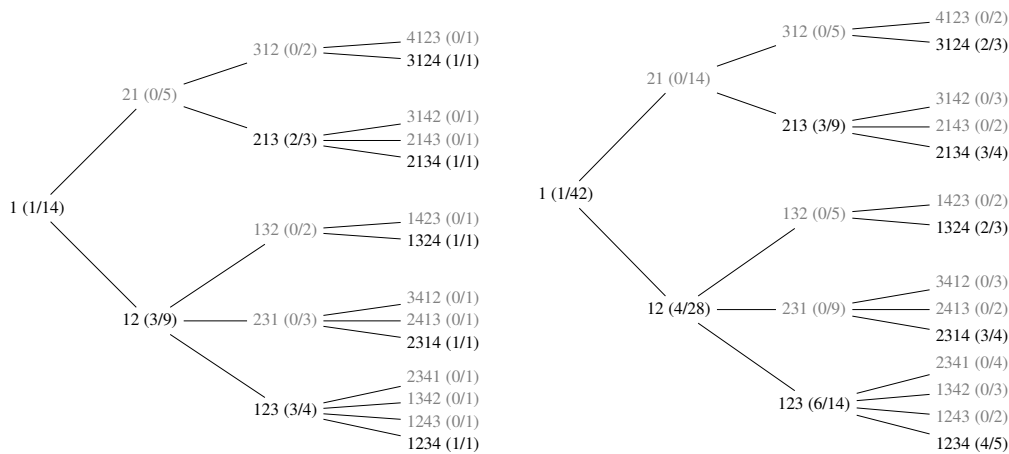

The following result describes the $\S_N(p)$ probabilities for arbitrary $N$.

\begin{theorem}\label{t:321_tree}
For the $321$-avoiding interview orders, the strike probabilities are
\[ \S_N(p) = \begin{cases}
        \S_{N-1}(\check{p}) & \mbox{ if $p$ is eligible and contains at least one inversion } \\
        \ & \ \\
        \frac{ { {N-1} \choose {k-1} } } { \frac{k+1}{N+1}  { {2N-k} \choose {N} } } & \mbox{ if $p = [12 \cdots k]$ } \\
    \end{cases} \]
where $\check{p}$ is the result of removing the value $1$ from $p$ (and flattening).
\end{theorem}
\begin{proof}
In any permutation $\pi$ with prefix $p$ that contains an inversion, the
position of the entry $1$ in $\pi$ is fixed:  it must occur in the prefix or
else it would create a $321$ instance when paired with the inversion in the
prefix.  Hence, removing the $1$ from this position is a bijection to the
corresponding subtree of $N-1$ prefixes.

Next, suppose $p$ has the form $[12 \cdots k]$ for some $k$.  We prove the
formulas for the denominator and numerator separately.  We claim the
denominators are ballot numbers.  To prove this, we show that they satisfy the
same recurrence.  Notice that there two possibilities for any $\pi$ of size $N$
having prefix $[12 \cdots k]$:  
\begin{itemize}
    \item If the element in the next position after the prefix is larger than the maximum value
in the prefix, then $\pi$ is counted by the prefix $[12 \cdots (k+1)]$ in $N$.  
    \item Otherwise, the next element is smaller and so $\pi$ is counted by the
        prefix $[12 \cdots (k-1)]$ in $N-1$ after applying the $\check{p}$ bijection.
\end{itemize}
The initial conditions are that we only have one element with prefix $[12 \cdots N]$
in $N$.  Hence, the denominators agree with the ballot numbers we defined in
Section~\ref{s:os}.

For the numerator, let $\pi$ be a $p$-winnable permutation with $p$ still
equal to $[12 \cdots k]$.  If $\pi$ has $N$ in position $k$ then everything
after $N$ must be increasing (to avoid $321$) so we just have to choose the
subset of values to place in the prefix.  These are counted by the binomial ${
{N-1} \choose {k-1}}$.
\end{proof}

In the prefix tree of rank $N$, let $\TB^\c(N,k)$ be the ``open'' subforest
lying under (but not including) $[12 \cdots k]$ and let $\TB(N,k)$ be the
``closed'' subtree lying under (and including) $[12 \cdots k]$.  Then we may
interpret the first part of Theorem~\ref{t:321_tree} as an isomorphism of
prefix ``forests'' (i.e. disjoint unions of trees).

\begin{corollary}\label{c:321isom}
We have a bijection
\[  \TB^\c(N,k) \cong \TB(N,k+1) \cup \bigcup_{i=1}^{k} \TB^\c(N-i, k+1-i) \]
where each nonzero strike probability on the left side occurs for the
isomorphic image on the right side.  In particular, we may obtain an optimal
strike set for the forest on the left side as the union of (the isomorphic
images of) optimal strike sets for each tree in the forest on the right side.
\end{corollary}
\begin{proof}
Note that the children of $[12 \cdots k]$ in the prefix tree consist of the
permutations of size $k+1$ that are increasing for the first $k$ entries and
end with some value $i = 1, 2, \ldots, k+1$.  Only the $[12 \cdots (k+1)]$
child ends with a left-to-right maximum so we apply the $\check{p}$ bijection
$i$ times to identify each of the other children with an increasing prefix from
some smaller rank.  Since none of these children are eligible in rank $N$,
though, we only transfer open subforests, not the node itself.  (Hence, we are
comitting an abuse of notation by including the nonincreasing children of $[12
\cdots k]$ themselves on the left side of the bijection; however, this does not
affect any of the nonzero strike probabilities or winning strategies.)

As explained in the proof of Theorem~\ref{t:321_tree}, the inverse of
$\check{p}$ simply inserts a new lowest entry into the position where $1$
appears in $p$.  In particular, the inverse image of a subforest $\TB^\c(N-i,
k+1-i)$ from the right side consists of the prefixes on the left side with
their $i$ smallest entries in fixed position.
\end{proof}

\begin{example}
Consider $N = 5$.  Iterating the bijection gives
\[  \TB^\c(5,1) \cong \TB(5,2) \cup \TB^\c(4, 1) \cong \TB(5,3) \cup \TB^\c(4,2) \cup \TB^\c(3,1) \cup \TB^\c(4, 1) \]
%\hspace{0.1in} \scriptstyle \text{\tiny i.e.} [1]^\c \cong [12] \cup [21]^\c \]
\[  \cong \TB(5,4) \cup \TB^\c(4,3) \cup \TB^\c(3,2) \cup \TB^\c(2,1) \cup \TB^\c(4,2) \cup \TB^\c(3,1) \cup \TB^\c(4, 1) \]
For example, $\TB^\c(3,2)$ here corresponds to the subforest under the prefix
$[1342]$ in $N = 5$.  (The bijection inserts value $1$ into the first position
and $2$ into the fourth position of each of the three prefixes under $[12]$ in $N
= 3$.)  We find that $[1234]$ is the optimal strike in $\TB(5,4)$ so we can stop here.
We may obtain the other strikes from smaller ranks using the bijection.
\end{example}

\bigskip
\subsection{Optimal strategy}
Let $\B^\c_N(p)$ denote the probability of success using an optimal strategy
for the interview orderings restricted to the $321$-avoiding open subforest
under (but not including) the prefix $p$.  The following corollary permits
$\B^\c$ to be computed recursively.

\begin{corollary}\label{c:defrec}
We have
\[  \B_N^\c(12 \cdots k) = \max(\B_N^\c(12 \cdots (k+1)), \S_N(12 \cdots (k+1))
) \oplus \bigoplus_{i=1}^{k} \B_{N-i}^\c( 12 \cdots (k+1-i) ). \]
Equivalently,
\[  \B_N^\c(12 \cdots k) = \B_{N-1}^\c(12 \cdots (k-1)) \oplus \left( \B_{N-1}^\c(12 \cdots k) - \max(\B_{N-1}^\c(12 \cdots k),
\S_{N-1}(12 \cdots k)) \right) \]
\[ \oplus \max(\B_N^\c(12 \cdots (k+1)), \S_N(12 \cdots (k+1))). \] 
\end{corollary}
\begin{proof}
This follows directly from the definitions and Corollary~\ref{c:321isom}.
\end{proof}

In particular, $\B_N^\c(1)$ gives the probability of success under the optimal
strategy for the entire collection of $321$-avoiding interview orders.  The
$\B_N^\c(12 \cdots k)$ numerators for $N \leq 16$ are shown in
Figure~\ref{f:tris}; the denominators are all ballot numbers (not shown).  

\begin{figure}[h]
{\tiny
\begin{tabular}{cllllllllllllllll}
$N \backslash k$ & $1$ & $2$ & $3$ & $4$ & $5$ & $6$ & $7$ & $8$ & $9$ & $10$ & $11$ & $12$ & $13$ & $14$ & $15$ \\
\hline \\
& & \\
2 &\bf 1$\ ^*$ \\
3 & 3 &\bf 1 \\
4 & 8 & 5 &\bf 1 \\
5 & 23 & 15 & 7 &\bf 1 \\
6 & 71 & 48 & 25 &\bf 9$\ ^*$ & 1 \\
7 & 229 & 158 & 87 & 39 &\bf 11 & 1 \\
8 & 759 & 530 & 301 & 143 & 56 &\bf 13 & 1 \\
9 & 2568 & 1809 & 1050 & 520 & 219 & 76 &\bf 15 & 1 \\
10 & 8833 & 6265 & 3697 & 1888 & 838 & 318 & 99 &\bf 17 & 1 \\
11 & 30797 & 21964 & 13131 & 6866 & 3169 & 1281 & 443 & 125 &\bf 19 & 1 \\
12 & 108613 & 77816 & 47019 & 25055 & 11924 & 5058 & 1889 & 608 &\bf 154$\ ^*$ & 21 & 1 \\
13 & 386804 & 278191 & 169578 & 91762 & 44743 & 19688 & 7764 & 2706 & 817 &\bf 186 & 23 & 1 \\
14 & 1389109 & 1002305 & 615501 & 337310 & 167732 & 75970 & 31227 & 11539 & 3775 & 1069 &\bf 221 & 25 & 1 \\
15 & 5024945 & 3635836 & 2246727 & 1244422 & 628921 & 291611 & 123879 & 47909 & 16682 & 5143 & 1368 & \bf 259 & 27 & 1 & \\
16 & 18292738 & 13267793 & 8242848 & 4607012 & 2360285 & 1115863 & 486942 & 195331 & 71452 & 23543 & 6861 & 1718 & \bf 300 & 29 & 1 & \\
\end{tabular} }
\caption{The first few rows of the $\B^\c$ triangle}\label{f:tris}
\end{figure}

Say that an entry $(N,k)$ is {\bf optimal} if $\S_N(12 \cdots k) > \B_N^\c(12
\cdots k)$.  Such an entry represents the root of a subtree for which the
optimal strike set is simply the prefix itself, so such a prefix is ``locally''
optimal.  The {\bf optimal boundary} in the $\B^\c$ triangle is the collection
consisting of the leftmost optimal entry from each row and we have highlighted
these in bold in Figure~\ref{f:tris}.  Such entries represent increasing
prefixes that are globally optimal, so appear in the optimal strike set for $N$.

We now prove an important structural result about these entries, namely, that
the optimal boundary proceeds by diagonal and vertical steps as $N$ increases.

\begin{theorem}\label{t:optimal}
If $(N,k)$ is optimal then so is $(N+1, k+1)$.  Also, if $(N,k)$ is not optimal then neither is $(N+1, k)$.
\end{theorem}
\begin{proof}
For the first statement in the theorem, note that it is enough to prove that
the ratios $\frac{\B^\c_N(12 \cdots k)}{\S_N(12 \cdots k)}$ are decreasing as
we move down diagonals; once $\frac{\B^\c}{\S} < 1$, it then remains so for all
subsequent entries along the same diagonal.  Equivalently, we show that the
ratios $\frac{\B^\c_{N-1}(12 \cdots (k-1))}{\B^\c_N(12 \cdots k)}$ are bounded
below by $\frac{\S_{N-1}(12 \cdots (k-1))}{\S_N(12 \cdots k)} =
\frac{k-1}{N-1}$.

For the remainder of this proof, we abuse notation to refer to the numerators
only; the denominators in each probability are ballot numbers which appear on
both sides of the inequality so can be canceled.  Following the ballot number
conventions, we also extend the $\B^\c$ triangle to include an extra column on
the left, defining $\B^\c_N(\emptyset) = \B^\c_N(1)$.

To avoid a profusion of closely related subscripts, we will let letters denote
positions in the $\B^\c$ and $\S$ triangles so each letter corresponds to a
particular $(N,k)$ pair as defined in Figure~\ref{f:co} .  If the letter is
unadorned, it represents the $\B^\c$ value at that position.  If it has a bar,
it represents the complementary term in a recurrence for $\B^\c$, namely the
$\max(\B^\c, \S)$ value at that position plus the $\B^\c-\max(\B^\c,\S)$ value
at the position just northwest.  Observe that this enables us to write 
\[ \po{x} = \po{b} + \pc{y} \] 
by Corollary~\ref{c:defrec}.  Finally, we represent the $\S$ values by capital
letters (to emphasize that these are simply binomial coefficients by
Theorem~\ref{t:321_tree}).  

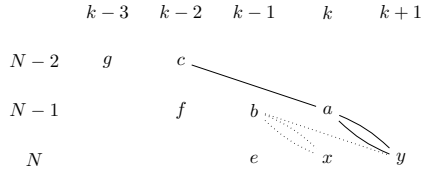
\begin{figure}[h]
\scalebox{0.65}{ \begin{tikzpicture}
\draw (1.5,5) node {$k-3$};
\draw (3,5) node {$k-2$};
\draw (4.5,5) node {$k-1$};
\draw (6,5) node {$k$};
\draw (7.5,5) node {$k+1$};
\draw (0,4) node {$N-2$};
\draw (0,3) node {$N-1$};
\draw (0,2) node {$N$};
\draw (1.5,4) node {$g$};
\draw (3,4) node (c) {$c$};
\draw (3,3) node {$f$};
\draw (4.5,3) node (b) {$b$};
\draw (6,3) node (a) {$a$};
\draw (4.5,2) node {$e$};
\draw (6,2) node (x) {$x$};
\draw (7.5,2) node (y) {$y$};
\path[-] (a)  edge [bend left=10] (y);
\path[-] (a)  edge [bend right=10] (y);
\draw (c) -- (a);
\draw [dotted] (b) -- (y);
\path[-] (b)  edge [dotted,bend left=10] (x);
\path[-] (b)  edge [dotted,bend right=10] (x);
\end{tikzpicture} }
\caption{Letter coordinates and key ratios for the induction argument}\label{f:co}
\end{figure}

Applying the defining recurrence (twice), we have
$\po{x} = \po{b} + \pc{y}$ and $\po{b} = \po{c} + \pc{a}$, so 
\begin{equation}\label{e:key}
    \frac{\po{b}}{\po{x}} = \frac{\po{b}}{\po{b}+\pc{y}} = \frac{1}{1+\frac{\pc{y}}{\po{b}}} =
    \frac{1}{1+\frac{\pc{y}}{\po{c}+\pc{a}}} 
    = \frac{1}{1+\frac{1}{\frac{\po{c}}{\pc{y}} + \frac{\pc{a}}{\pc{y}} } } 
    = \frac{1}{1+\frac{1}{ \frac{\pc{a}}{\pc{y}} \left( \frac{\po{c}}{\pc{a}} + 1\right) } }.
\end{equation}
Now, if we produce lower bounds for $\frac{\po{c}}{\pc{a}}$ and
$\frac{\pc{a}}{\pc{y}}$ then this expression gives a lower bound for
$\frac{\po{b}}{\po{x}}$.  So, {\em we claim the following lower bounds}, to be proved
by induction:
\begin{equation}\label{e:ind}
    \frac{\po{b}}{\po{x}} \geq \frac{B}{X}, \hspace{0.2in} \frac{\po{b}}{\pc{y}} \geq
\frac{B}{X-B}, \text{ and } \hspace{0.2in} \frac{\pc{b}}{\pc{x}} \geq
\frac{F-G}{E-F}.
\end{equation}
If valid, we say that an inequality holds {\bf at} $b$, the entry appearing in the
numerator on the left side.  Recall that the first bound is our primary goal in
this proof (but we evidently require the others to facilitate it).

For the base case, we solve the recurrences to find explicit formulas along the
rightmost diagonals:  we find that $\B^\c_N(N-1) = 1$ and $\B^\c_N(N-2) = 2N-3$
for all $N \geq 3$.  The bar value $\pc{y}$ is $2$ at position $(N,N-1)$ for
all $N \geq 3$ and $\pc{x}$ is $3N-7$ at position $(N,N-2)$ for all $N \geq 7$.

The binomial formula for $\S$ from Theorem~\ref{t:321_tree} yields
\[ \frac{B}{X} = \frac{k-1}{N-1}, \hspace{0.2in} \frac{B}{X-B} =
\frac{k-1}{N-k}, \hspace{0.2in} \frac{F-G}{E-F} = \frac{(k-2)}{(N-2) } =
\frac{C}{B}. \]

Then we can verify each of the claimed inequalities from (\ref{e:ind}) directly for $k = N-2$, $N \geq 7$:  
\[ \frac{\po{b}}{\po{x}} = \frac{2N-5}{2N-3} \geq \frac{N-3}{N-1} = \frac{B}{X}, \]
\[ \frac{\po{b}}{\pc{y}} = \frac{2N-5}{2} \geq \frac{N-3}{2} = \frac{B}{X-B}, \text{ and } \]
\[ \frac{\pc{b}}{\pc{x}} = \frac{3N-10}{3N-7} \geq \frac{N-4}{N-2} = \frac{F-G}{E-F}. \]

Now assume that we have all three of the inequalities from (\ref{e:ind}) at
entries along the current diagonal including the entry $c$ and at all entries on
diagonals to the right including the entry $a$.  We now derive each of the
bounds at $b$.

By the induction hypothesis,
\[ \frac{\pc{a}}{\pc{y}} \geq \frac{B-C}{X-B} \ \ \text{  and  } \ \  \frac{\po{c}}{\pc{a}} \geq \frac{C}{B-C}. \]
Substituting these into our expression from Equation~(\ref{e:key}), we obtain
\[ \frac{\po{b}}{\po{x}} = \frac{1}{1+\frac{1}{ \frac{\pc{a}}{\pc{y}} \left( \frac{\po{c}}{\pc{a}} + 1\right) } } \geq  
   \frac{1}{1+\frac{1}{ \frac{B-C}{X-B} \left( \frac{C}{B-C} + 1\right) } } = 
   \frac{1}{1+\frac{1}{ \frac{B-C}{X-B} \left( \frac{B}{B-C} \right) } } = 
   \frac{1}{1+\frac{1}{ \frac{B}{X-B} } }.  \]
When we simplify this expression, we find that this is precisely equal to $\frac{B}{X}$, as required.
Moreover, this calculation already verifies the second bound as
\[ \frac{\po{b}}{\pc{y}} = \frac{\pc{a}}{\pc{y}} \left( \frac{\po{c}}{\pc{a}} +
1\right) \geq \frac{B-C}{X-B} \left(  \frac{C}{B-C} + 1\right) = \frac{B}{X-B}. \]

Finally, we evaluate $\frac{\pc{b}}{\pc{x}}$ in cases.  Recall,
\[ \pc{b} = \max(\B^\c, \S)(b) + \B^\c(c) - \max(\B^\c,\S)(c) \]
The argument up to here already shows if $\max(\B^\c, \S) = \B^\c$ for entry
$c$, $b$, or $x$, then it remains so for all entries above to the northwest.  So the cases for
\[ \frac{\pc{b}}{\pc{x}} = \frac{\max(\B^\c, \S)(b) + \B^\c(c) - \max(\B^\c,\S)(c)}{\max(\B^\c, \S)(x) + \B^\c(b) - \max(\B^\c,\S)(b)}  \]
are:  none of $c$, $b$, or $x$ have $\max = \B^\c$; $c$ alone does; $b$ and $c$ do; or all three do.

For the first case, we claim that
\[ \frac{\pc{b}}{\pc{x}} = \frac{B + c - C}{X + b - B} \geq \frac{C}{B} = \frac{F-G}{E-F} \]
because
$\frac{c}{b} \geq \frac{C}{B}$ and $\frac{B}{X} \geq \frac{C}{B}$ so $B(B+c) \geq C(X+b)$ whence
\[ B(B+c-C) \geq C(X+b-B) \]
yielding the desired inequality.

For the other three cases, we have
$\frac{B + c - c}{X + b - B} = \frac{B}{X+b-B}  \geq \frac{B}{X}$, 
$\frac{b + c - c}{X + b - b} = \frac{b}{X} \geq \frac{B}{X}$, and
$\frac{b + c - c}{x + b - b} = \frac{b}{x}$.
For the last case, note that $\frac{\po{b}}{\po{x}} \geq \frac{B}{X}$ by the
induction argument to this point.  So in each case we have
$\frac{\pc{b}}{\pc{x}} \geq \frac{B}{X}$ and since $(k-1)(N-2) \geq (N-1)(k-2)$
for $k \leq N$, we obtain $\frac{B}{X} \geq \frac{F-G}{E-F}$ as desired.

%Incidentally: \[ \frac{c}{b} \geq \frac{B}{X}. \]

The induction proceeds along each row from right to left.  Since
$\B^\c_N(\emptyset) = \B^\c_N(1)$, once we prove that $\frac{b}{x} \geq
\frac{B}{X}$ for $b = (N-1,1)$, we automatically get that $\frac{f}{\pc{x}} =
\frac{b}{x} \geq \frac{B}{X} = \frac{F}{E-F}$.  Thus, the induction may proceed
to the next row.

The proof of the second statement in the theorem now follows directly.
Continuing our notation, we have
\[ 1 = \frac{b + \pc{y}}{x} \geq \frac{b + a}{x} \geq \frac{B}{X} + \frac{a}{x} \]
because $\pc{y} = a + \left(\max(\B^\c, \S)(y) - \max(\B^\c,\S)(a)\right)$ and by the inequalities we already proved.
Hence,
\[ \frac{a}{x} \leq 1 - \frac{B}{X} = \frac{X-B}{X} = \frac{1}{\left(\frac{B}{X-B}\right)}
    \left(\frac{B}{X}\right) = \frac{N-k}{N-1} = \frac{A}{X}. \]
Thus, if $a$ is not optimal then $1 \leq \frac{a}{A} \leq \frac{x}{X}$ and so
$x$ is not optimal either.
\end{proof}

\begin{definition}
We say that the {\bf value-saturated left-to-right maxima} of a prefix $p = p_1
\cdots p_k$ are the largest subset of left-to-right maxima in $p$ whose values
form an interval $\{k-i+1, k-i+2, \cdots, k-1, k\}$ for some $i$.  
\end{definition}

Given a node $p = [p_1 p_2 \cdots p_k]$ from the prefix tree of $321$-avoiding
permutations of size $N$, we refer to $N$ as the {\bf rank} of $p$ and $k$ as
the {\bf size} of $p$.  Let $\sigma(i)$ be the minimal $k$ such that $(N,k)$ is
optimal among the entries where $N-k = i$.  That is, $\sigma(i)$ is the column
containing the $i$th vertical step along the optimal boundary of the $\B^\c$
table.  We know that $\sigma(i)$ is a well-defined weakly increasing
function by Theorem~\ref{t:optimal}.

We say $p$ is {\bf selected by the threshold} if 
\begin{equation}\label{e:tie}
\#\text{ value-saturated left-to-right maxima in } p \geq \sigma(\text{rank}(p) - \text{size}(p)). 
\end{equation}

\begin{lemma}\label{l:bijok}
We have that $p$ is selected by the threshold if and only if $\check{p}$ is selected by the threshold.
\end{lemma}
\begin{proof}
Applying the $\check{p}$ bijection from Theorem~\ref{t:321_tree} removes the
lowest entry from $p$ and reduces the rank by $1$.  This does not change the
right side of the threshold inequality~(\ref{e:tie}).  The only way that removing
the lowest entry could change the number of value-saturated left-to-right
maxima in $p$ is if we removed the first entry of a prefix with the form $[12 \cdots i]$.  But
this is not possible because we require at least one inversion in $p$ in order
to apply the bijection.  Hence, we do not change the left side of the threshold
inequality either.
\end{proof}

\begin{theorem}\label{t:ssvslrm}
For any $N$ and $k$, we have that $p$ is included in the optimal strike set for
$\TB^\c(N,k)$ if and only if 
\begin{itemize}
\item $p$ is eligible,
\item $p$ is selected by the threshold, and
\item no proper prefix of $p \in \TB^\c(N,k)$ is selected by the threshold.
\end{itemize}
\end{theorem}
\begin{proof}
We argue by strong induction on $N$.  For $N = 1$ and $N = 2$, we have that the
prefix $[1]$ is an optimal strike and it is selected by the threshold in each
case.  So suppose the result holds for all ranks $M < N$.

Recalling Corollary~\ref{c:321isom}, we may view $\TB^\c(N,k)$ as a disjoint union.
Any prefix from $\bigcup_{i=1}^{k} \TB^\c(N-i, k+1-i)$ will be selected by the
threshold if and only if it is optimal, by our induction hypothesis and
Lemma~\ref{l:bijok}.

So it suffices to show that the prefixes from $\TB(N,k+1)$ conform to the
threshold.  Let $j$ be the leftmost optimal entry in row $N$ of the $\B^\c$
table.  If $k+1 \leq j$ then $[12 \cdots j]$ is in the optimal strike set by
definition.  Otherwise, we claim that $[12 \cdots (k+1)]$ is in the optimal
strike set.  This follows from the fact that the strike numerators for
increasing prefixes are binomial coefficients so this prefix must have the
largest strike numerator in the subtree under $[12 \cdots (k+1)]$ (or else we
contradict $k+1 > j$ because the set of optimal entries on each row of the
$\B^\c$ table is an interval by Theorem~\ref{t:optimal}).

Now, since the number of value-saturated left-to-right maxima in an increasing
prefix is equal to its size, we find that an increasing prefix is selected by
the threshold precisely when it corresponds to an optimal entry $(N,k)$ in the
$\B^\c$ table.  Thus, the prefixes from $\TB(N,k+1)$ conform to the threshold.
\end{proof}

\begin{corollary}\label{c:321_strat}
For $321$-avoiding interview orders, the optimal strategy is a threshold strike
strategy using the number of value-saturated left-to-right maxima as the
statistic.  Moreover, a single threshold function works simultaneously for all $N$.
\end{corollary}
\begin{proof}
This follows by applying Theorem~\ref{t:ssvslrm} to $\B^\c(N,1)$.  
\end{proof}

To play the optimal strategy on a given interview ordering $\pi$, reject
candidates until we are at some eligible prefix flattening, $\pi|_{[k]}$, where
we have seen $\sigma(N-k)$ value-saturated left-to-right maxima.  Then select
the $k$th candidate.  From small values of $N$, we can compute the beginning of
the optimal threshold function precisely.  The first few values are:
\[ \sigma(N) = \sigma(N-1) = 1, \sigma(N-2) = 4, \sigma(N-3) = 9, \sigma(N-4) = 16, \ldots. \] 
These rules say that by the time you get to interview $k = N$ or $k = N-1$, you
should always accept the $k$th interview candidate.  However, when $N$ is large
enough to permit it and you have seen an ``unusually high'' number of
value-saturated left-to-right maxima, it can also be optimal to accept an
earlier candidate.

The complete list of rules (compiled for all $N \leq 10000$) indicate that the
threshold function has the form $\sigma(N-i) = i^2$ for all $i \geq 1$.  We do know
that the rules we computed for these ``small'' values of $N$ remain in force for
all $N$ by Theorem~\ref{t:optimal}.  It would be interesting to obtain the
asymptotic threshold function and probability of success precisely.

\subsection{An asymptotic lower bound}

In this section we show how to compute the probability of success
for the particular threshold strategy defined by using the first four rules,
$\sigma(N) = \sigma(N-1) = 1$, $\sigma(N-~2) = 4$, and $\sigma(N-3) = 9$,
from the optimal strategy.  This threshold strategy has an associated triangle
of ``interior'' probabilities that we denote $\B^\c_{(1,4,9)}(N,k)$.  The
only difference between this triangle and the $\B_N^\c(12 \cdots k)$ triangle
from the last subsection is that the optimal boundary (i.e. the leftmost
entries where $\S > \B^\c_{(1,4,9)}$) follows the $k=N-3$ diagonal forever
and so all entries on or left of the $k=N-5$ diagonal are obtained from the
recurrence 
\[ \B^\c_{(1,4,9)}(N,k) = \B^\c_{(1,4,9)}(N-1,k-1) + \B^\c_{(1,4,9)}(N,k+1). \]  
This is the same recurrence that is satisfied by the ballot numbers, and it
turns out we can write $\B^\c_{(1,4,9)}$ as a linear combination of ``shifted''
ballot numbers.  This allows us to compute the first entry in each row
precisely, which gives the total probability of success.

\begin{definition}
Define the {\bf $i$-shifted ballot triangle $\C_{i}(N,k)$} to be the result of
replacing $N$ by $N-i$ in the ballot number formula from Section~\ref{s:os}
(where we interpret any binomial coefficients with negative indices as zero).
\end{definition}

\begin{figure}[t]
{\tiny
\begin{tabular}{lllllllllllllllllll}
$N \backslash k$ & $1$ & $2$ & $3$ & $4$ & $5$ & $6$ & $7$ & $8$ & $9$ & $10$ & $11$ & $12$ & $13$ & $14$ \\
\hline \\
2 & 4 & \\
3 & -1 & 4 & \\
4 & 2 & 3 & 4 & \\
5 & 13 & 9 & \bf 7 & 4 & \\
6 & 151 & 33 & 20 & 11 & 4 & \\
7 & 164 & 219 & 68 & 35 & 15 & 4 & \\
8 & 764 & 505 & 341 & 122 & 54 & 19 & 4 & \\
9 & \bf 2568 & \bf 1809 & 1045 & 540 & 199 & 77 & 23 & 4 & \\
10 & \bf 8833 & \bf 6265 & \bf 3697 & \bf 1888 & 843 & 303 & 104 & 27 & 4 & \\
11 & \bf 30797 & \bf 21964 & \bf 13131 & \bf 6866 & \bf 3169 & \bf 1281 & 438 & 135 & 31 & 4 & \\
12 & \bf 108613 & \bf 77816 & \bf 47019 & \bf 25055 & \bf 11924 & \bf 5058 & \bf 1889 & \bf 608 & 170 & 35 & 4 & \\
13 & \bf 386804 & \bf 278191 & \bf 169578 & \bf 91762 & \bf 44743 & \bf 19688 & \bf 7764 & \bf 2706 & \bf 817 & 209 & 39 & 4 & \\
14 & \bf 1389109 & \bf 1002305 & \bf 615501 & \bf 337310 & \bf 167732 & \bf 75970 & \bf 31227 & \bf 11539 & \bf 3775 & \bf 1069 & 252 & 43 & 4 & \\
15 & \bf 5024945 & \bf 3635836 & \bf 2246727 & \bf 1244422 & \bf 628921 & \bf 291611 & \bf 123879 & \bf 47909 & \bf 16682 & \bf 5143 & \bf 1368 & 299 & 47 & 4 & \\
16 & \bf 18292738 & \bf 13267793 & \bf 8242848 & \bf 4607012 & \bf 2360285 & \bf 1115863 & \bf 486942 & \bf 195331 & \bf 71452 & \bf 23543 & \bf 6861 & \bf 1718 & 350 & 51 & 4 & \\
\end{tabular}  
}
\caption{The linear combination of ballot numbers that agrees with $\B^\c_{(1,4,9)}$}\label{f:lcb}
\end{figure}

\begin{theorem}\label{t:lcb}
We have that $\B^\c_{(1,4,9)}(N,k)$ agrees with the triangle 
\[ 4 \C_1(N,k) - 9 \C_2(N,k) + 2 \C_4(N,k) + 105 \C_5(N,k) - 206 \C_6(N,k) + 95 \C_7(N,k) - 5 \C_8(N,k). \]
for all $N \geq 11$ and $1 \leq k \leq N-5$.
\end{theorem}
\begin{proof}
As we have observed, all of the triangles under discussion satisfy the recurrence 
\[ X(N,k) = X(N-1,k-1) + X(N,k+1) \]
at least in the region of $(N,k)$ values lying on or to the left of the
$k=(N-5)$th diagonal.  But this can be translated to recurrences for the
entries along each particular diagonal.

Let $x_N = X(N,N-k)$ be the sequence of entries along a particular diagonal.
When $k = 2$, we have $x_{N+1} = x_{N} + c$.  Successively replacing each $x_N$
with its difference $x_N - x_{N-1}$ we obtain a recurrence for the next diagonal
to the left.  For $k = 3$, for example, we obtain $(x_{N+1}-x_N) = (x_N -
x_{N-1}) + c$ which reduces to $x_{N+1} = 2x_N - x_{N-1} + c$.  In general, the
recurrence for entries along the $k$th diagonal will have degree $k$ (and its
coefficients will be alternating binomial coefficients; see \cite{concrete}).

Hence, once we have agreement between $\B^\c_{(1,4,9)}$ and the linear
combination of shifted ballot numbers for six terms along the $k = N-5$
diagonal, we must have agreement forever along this diagonal and therefore for
all entries along diagonals to the left.  A finite computation illustrated in
Figure~\ref{f:lcb} shows that this indeed occurs for $N \geq 11$ and $k \leq
N-5$.
\end{proof}

\begin{corollary}\label{c:321_main}
The asymptotic probability of success under the optimal strategy for the
$321$-avoiding interview orders is at least $32983/65536$ which is about
$0.5032806396484$.
\end{corollary}
\begin{proof}
In the linear combination of shifted ballot numbers from Theorem~\ref{t:lcb}, set $k =
1$, divide by the $N$th Catalan number, and take the limit as $N \rightarrow
\infty$.  We obtain 
\[ 4 \left(\frac{1}{4}\right)^1 - 9 \left(\frac{1}{4}\right)^2 + 2
\left(\frac{1}{4}\right)^4 + 105 \left(\frac{1}{4}\right)^5 - 206
\left(\frac{1}{4}\right)^6 + 95 \left(\frac{1}{4}\right)^7 - 5 \left(\frac{1}{4}\right)^8. \]
\end{proof}

\subsection{Other strategies}

We have also investigated positional and trigger strategies for $321$-avoiding
interview orders.  Although they are not generally optimal, it may be convenient
to compare them more directly with the classical best choice game analysis.
Positional strategies are also simpler to implement and require less memory.  In
this subsection, we briefly outline the main results.

The $321$-avoiding permutations are completely determined by the values and
positions of their left-to-right maxima (as the complementary entries must be
increasing to avoid $321$).  These left-to-right maxima may be encoded by Dyck
paths on an $N \times N$ grid lying above the diagonal.  The interviews that
are $k$-winnable (where $k$ represents a position to transition from rejection
to hiring) then correspond to Dyck paths whose next to last horizontal segment
passes through column $k$.  These can be counted recursively to obtain a
success probability in the form of a linear combination of ratios of Catalan
numbers.  For example, the positional strategy of transitioning after $k = N-3$
has a success probability of 
\[ \frac{3 \C_{N-1} - 4 \C_{N-2} - \C_{N-3}}{\C_N}. \]
The main result in this direction is that for all $N > 8$, transitioning from
rejection to hiring after $k = N-3$ turns out to be the optimal positional
strategy and has limiting value $\frac{31}{64} = 0.484375$.  Further details
are given in \cite{fowlkes}.

More generally, we considered trigger strategies for the $321$-avoiding
interviews.  Replacing Theorem~\ref{t:321_tree}, we have
\begin{theorem}
For the $321$-avoiding interview orders, we have
\[ \T_N(p) = \begin{cases}
\T_{N-1}(\check{p}) & \mbox{ if $p$ contains at least one inversion } \\
\ & \ \\
\frac{ k { {N-1} \choose {k+1} } + { {N-1} \choose {k} } } { \frac{k+1}{N+1}  { {2N-k} \choose {N} } } & \mbox{ if $p = [12 \cdots k]$\ (with $k=0$ corresponding to $p = \emptyset$) } \\
\end{cases} \]
where $\T_N(p)$ is the probability of winning (restricted to the subtree under
$p$) if $p$ is included in the trigger set.
\end{theorem}
One can define an analogous $\B_N^\c$ triangle with slightly simpler recurrence 
\[ \B_N^\c(12 \cdots k) = \B_{N-1}^\c(12 \cdots (k-1)) \oplus \max(\B_N^\c(12 \cdots (k+1)), \T_N(12 \cdots (k+1)) ). \]
The optimal trigger sets can be translated to a statistical trigger strategy
based on value-saturated left-to-right maxima, proved along the same lines as we
have done for our strike strategy.  The optimal trigger threshold function grows
similarly to a quadratic but does not seem to satisfy a simple formula.  The
first few values are 
\[ \sigma(N-2) = 1, \sigma(N-3) = 3, \sigma(N-4) = 8, \sigma(N-5) = 15, \sigma(N-6) = 25, \sigma(N-7) = 36, \ldots. \] 
For $N \geq 12$, it turns out that the optimal trigger strategy is not optimal
overall.  However, a lower bound (proved similarly to Theorem~\ref{t:lcb}) for the
asymptotic probability of success using the optimal trigger strategy is
$8239/16384$, which is about $0.50286865$.

\subsection{Bijection for $312$-avoiding}

In this subsection, we explain why the $321$-avoiding prefix tree is isomorphic
to the $312$-avoiding prefix tree.  This is essentially ``West's bijection'' of
generating trees described in \cite{west2} and \cite{ck_bij} although some of
the position/value conventions are different; we give a self-contained account
here for completeness.

Given a prefix permutation of size $N-1$, there are potentially $N$ children in
the prefix tree, each corresponding to a value in the last position which we
refer to as {\bf index} of the child.  The other values and positions in the
child are then completely determined by the parent prefix.

These indices are determined by inversions in the parent permutation.  We will
refer to an inversion by its values which we denote by $(b>a)$.

\begin{lemma}\label{l:3ch}
Fix a prefix permutation $\pi$ of size $N-1$ in the $321$-avoiding prefix tree.  
Then, the indices for the children of $\pi$ are
\[ \{1, 2, \ldots, N\} \setminus \{j : j \leq a \text{ for an inversion }
(b>a) \text{ in } \pi\}. \]
(So it suffices to consider the inversion with the largest minimal value.)

Fix a prefix permutation $\pi$ of size $N-1$ in the $312$-avoiding prefix tree.  
Then, the indices for the children of $\pi$ are
\[ \{1, 2, \ldots, N\} \setminus \{j : a < j \leq b \text{ for an inversion } (b>a) \text{ in } \pi\}. \]
\end{lemma}
\begin{proof}
It is straightforward to verify that these conditions directly encapsulate the pattern avoidance criteria.
\end{proof}

\begin{corollary}
The child indices form a nested decreasing sequence of subsets along any path in
the prefix tree.
\end{corollary}

\begin{theorem}
The prefix tree for $321$-avoiding permutations is isomorphic to the prefix
tree for $312$-avoiding permutations.
\end{theorem}
\begin{proof}
We will actually claim a little more in order to establish the isomorphism.
First, we claim that the number of children for a permutation $\pi$ in the
$321$-avoiding prefix tree is the same as the number of children in the
isomorphic node of the $312$-avoiding prefix tree.  This means the denominators
of the strike probabilities remain the same.  Second, we claim that the positions (but
not necessarily the values!) of the left-to-right maxima of $\pi$ in the
$321$-avoiding prefix tree are equal to the positions of the left-to-right
maxima of the isomorphic node of the $312$-avoiding prefix tree.  This means
the position of $N$, and hence winnability, is preserved so the numerators of
the strike probabilities remain the same.

We work by induction on permutation size.  The claims are true for permutations
of size $<3$, establishing a base case.

Suppose that $\pi$ is a $321$-avoiding permutation of size $N-1$ with child
indices 
\[ c_1, c_2, \ldots, c_{m-1} = N-1, c_m = N \]
(arranged increasingly).  By induction, there exists a corresponding
$312$-avoiding permutation $\widetilde{\pi}$ with the same number of children,
say 
\[ \widetilde{c}_1 = 1, \widetilde{c}_2, \ldots, \widetilde{c}_m = N \]
(also arranged increasingly), and having the same positions for left-to-right
maxima as $\pi$.

We now show that each of the children of $\pi$ also has a corresponding element
in the $312$-avoiding prefix tree.  The child $c_m = N$ corresponds to the child
$\widetilde{c}_m = N$ and each child has the same indices as its parent since no
new inversions are created.  Hence, the first claim is satisfied by induction.
Also, $N$ will be a new left-to-right maximum, so the second claim is satisfied
by induction.

Otherwise, we claim that the child $c_i$ will correspond to the child
$\widetilde{c}_{m-i}$.  To see this, note that the value $c_i$ will form the
minimal entry of a new inversion $(N > c_i)$ in the child permutation.  Hence,
by Lemma~\ref{l:3ch}, the child $c_i$ will need to remove $c_1, c_2, \ldots,
c_i$ from its set of indices.

Similarly for the $312$-avoiding tree, the value $\widetilde{c}_{m-i}$ forms the
minimal entry of a new inversion $(N > \widetilde{c}_{m-i})$ so the child
$\widetilde{c}_{m-i}$ will need to remove $\widetilde{c}_{m-i+1},
\widetilde{c}_{m-i+1}, \ldots, \widetilde{c}_{m}$ by Lemma~\ref{l:3ch}, which is
also a total of $i$ indices removed.  Also, one new index $N+1$ will be added
for each child.

Hence, the first claim is satisfied by induction.  Also, neither of the entries
in position $N$ will be a left-to-right maximum, so the second claim is
satisfied by induction.
\end{proof}

\begin{example}
    The correspondence begins as shown in Figure~\ref{f:3n4}.
\end{example}

\begin{figure}[t]
    {\tiny \begin{tabular}{ll|ll}
321-avoiding & & 312-avoiding & \\
\hline 
213 & & 213 & \\
 & 3142 & & 2143 \\
 & 2143 & & 3241 \\
 & 2134 & & 2134 \\
312 & & 321 & \\
 & 4123 & & 4321 \\
 & 3124 & & 3214 \\
132 & & 231 & \\
 & 1423 & & 3421 \\
 & 1324 & & 2314 \\
231 & & 132 & \\
 & 3412 & & 1432 \\
 & 2413 & & 2431 \\
 & 2314 & & 1324 \\
123 & & 123 & \\
 & 2341 & & 1243 \\
 & 1342 & & 1342 \\
 & 1243 & & 2341 \\
 & 1234 & & 1234 \\
\end{tabular} }
\caption{Correspondence for $N = 4$}\label{f:3n4}
\end{figure}

%%%%%%%%%%%%%%%%%%%%%%%%%%%%%%%%%%%%%%%%%%%%%%%%%%%%%%%%%%%%%%%%%%%%%
%  Section
%%%%%%%%%%%%%%%%%%%%%%%%%%%%%%%%%%%%%%%%%%%%%%%%%%%%%%%%%%%%%%%%%%%%%
\bigskip
\section{The remaining size three patterns}\label{s:conclusions}

The $123$-avoiding and $213$-avoiding interview orders give distinct best
choice problems but have distributions that are essentially decreasing so the
optimal strategy is to always to choose the next left-to-right maximum after
the first entry.  Although these are not so interesting from the game
perspective, they complete our analysis of all the size $3$ patterns.  For the
following results, we continue to let $\C(N,k)$ denote the ballot numbers.

\begin{proposition}
For the $123$-avoiding interviews, we have 
\[ \S_N(p) = \begin{cases}
\C(N-1,k-1) / \C(N-1,k-1) & \mbox{ if $p = (k-1)\ (k-2)\ \cdots 2\ 1\ k$ and $2 \leq k \leq N$ } \\
\frac{\C(N-1)}{\C(N)} & \mbox{ if $p = 1$ } \\
0 & \mbox{ otherwise. } \\
\end{cases} \]
Hence, the optimal strategy is to accept the second left-to-right maximum.
This succeeds with asymptotic probability $3/4$.
\end{proposition}
\begin{proof}
Observe that a $123$-avoiding permutation can have at most two left-to-right
maxima, and $\S(p)$ is $0$ unless $p$ has the form of a decreasing sequence
followed by a left-to-right maximum.
That is, $p = (k-1)(k-2) \cdots 1 k$ with one prefix for each $1 \leq k \leq N$.

By rejecting the first interview candidate, we can capture all of the other
eligible strikes.  When we add these, we obtain total probability of success
\[ \frac{\C(N-1,N-1)+\C(N-1,N-2)+ \cdots +\C(N-1,1)}{\C(N)} = \frac{\C(N,2)}{\C(N)} = \frac{3N(N-1)}{2N(2N-1)} \]
which is more than the $\frac{\C(N-1)}{\C(N)}$ as we would obtain by accepting
$p = 1$.
\end{proof}

\begin{proposition}
For the $213$-avoiding interviews, we have 
\[ \S_N(p) = \begin{cases}
\frac{\C(N-1,k-1)}{\C(N,k)} & \mbox{ if $p = 12 \cdots k$ and $1 \leq k \leq N$ } \\
0 & \mbox{ otherwise. } \\
\end{cases} \]
Hence, the optimal strategy is accept the first or second left-to-right maximum.
This succeeds with asymptotic probability $1/4$.
\end{proposition}
\begin{proof}
Observe that any $213$-avoiding permutation that ends in a left-to-right maximum
must be increasing.  So, $\S_N(p) = 0$ if $p$ is not increasing.  Otherwise, $p =
12 \cdots k$, so $\S_N(p) = \frac{\C(N-1,k-1)}{\C(N,k)}$ (as the ballot numbers
count the total number of $213$-avoiding permutations under a given increasing
prefix, so removing the value $N$ gives a bijection to smaller rank).  These
prefixes all contain each other so we should pick the one with largest number of
wins.  As the numerators are decreasing, it is optimal to choose the first or
second one.
\end{proof}

%%%%%%%%%%%%%%%%%%%%%%%%%%%%%%%%%%%%%%%%%%%%%%%%%%%%%%%%%%%%%%%%%%%%%
%  Acknowledgements
%%%%%%%%%%%%%%%%%%%%%%%%%%%%%%%%%%%%%%%%%%%%%%%%%%%%%%%%%%%%%%%%%%%%%
\bigskip
\section*{Acknowledgements}

This project was supported by the James Madison University Program of Grants
for Faculty Assistance.  An early version of this research was also the subject
of an undergraduate research program mentored by the author at James Madison
University; we are grateful to Aaron Fowlkes for his keen contributions to that
work as well as to Stephen Lucas, Edwin O'Shea, and Laura Taalman for their
interest and helpful comments.  We also thank the anonymous referees, whose
comments improved this manuscript.

%%%%%%%%%%%%%%%%%%%%%%%%%%%%%%%%%%%%%%%%%%%%%%%%%%%%%%%%%%%%%%%%%%%%%
% ==========   Bibliography
%%%%%%%%%%%%%%%%%%%%%%%%%%%%%%%%%%%%%%%%%%%%%%%%%%%%%%%%%%%%%%%%%%%%%

%\bibliographystyle{alpha}
%\bibliographystyle{amsalpha}
%\bibliography{bestchoice}

\providecommand{\bysame}{\leavevmode\hbox to3em{\hrulefill}\thinspace}
\providecommand{\MR}{\relax\ifhmode\unskip\space\fi MR }
\providecommand{\MRhref}[2]{%
  \href{http://www.ams.org/mathscinet-getitem?mr=#1}{#2}
}
\providecommand{\href}[2]{#2}

\end{document}